\documentclass[14pt]{article}

\usepackage{a4wide,amsmath,pxfonts}
\usepackage{amsfonts,multicol}
\usepackage{euscript}
\normalfont

\let\orgnonumber=\nonumber\usepackage{mathenv}\let
\nonumber=\orgnonumber

\allowdisplaybreaks

\usepackage{multirow}
\usepackage{graphicx}
\usepackage[colorlinks, citecolor=blue]{hyperref}

\newtheorem{theorem}{\indent Theorem}[section]
\newtheorem{lemma}{\indent Lemma}[section]
\newtheorem{corollary}{\indent Corollary}[theorem]
\newtheorem{proposition}{\indent Proposition}[section]
\newtheorem{definition}{\indent Definition}[section]

\newenvironment{proof}                    
       {\par\indent{\bf Proof. }} 
	        {\hfill $\scriptstyle\blacksquare$}

\begin{document}

\centerline{{\large{\bf Reflective anisotropic hyperbolic lattices
of rank $4$}}}

\vspace*{0.2cm}

\begin{center}
{\bf Bogachev N.V.$^{a}$}

%\noindent
\small{$^{a}$%{\it 
Department of Mathematics and Mechanics,
Lomonosov Moscow State University, \\ 119991 Leninskie Gory, Moscow, Russia%}
}
\end{center}

\vspace*{0.2cm}

%\begin{center}
%\parbox[r]{15cm}{
%\noindent
%{\scshape Abstract}

%\begin{small}
%\textit{
\begin{abstract}
A hyperbolic lattice is called
\textit{$1.2$-reflective} if   
the subgroup of its automorphism group 
generated by all $1$- and $2$-reflections is of finite index.  
The main result of this article is a complete
classification of $1.2$-reflective maximal anisotropic lattices of rank $4$.
\end{abstract}
%}
%\end{small}
%}
%\end{center}
%\vspace*{0.2cm}


\tableofcontents

\section{Introduction}\label{sec}

By definition, a \textit{quadratic lattice} is a free Abelian group with an
integral symmetric bilinear form called a scalar product. A quadratic lattice $L$ is
called \textit{Euclidean}
if its scalar product is positively definite and it is called \textit{hyperbolic}
if its scalar product is a form of signature $(n, 1)$.

Let $L$ be a hyperbolic lattice. Then
$V =  L \otimes \mathbb{R} = \mathbb{E}^{n,1}$ is a Minkowski space. Hence the
group $O(L)$
of automorphisms of the lattice $L$ is a lattice in the pseudoorthogonal
group $O(V)$.
One of the connected components of the hyperboloid
$$
\{ x \in \mathbb{E}^{n,1}\colon (x, x) = -1\}
$$
 will be considered as the $n$-dimensional Lobachevsky space $\mathbb{L}^n$.
In this case, the group of
motions of $\mathbb{L}^n$ is a subgroup $O'(V)$ of index $2$ in $O(V)$.
It consists of all
transformations leaving invariant each connected component of the hyperboloid.
The planes in the vector model of the Lobachevsky space are non-empty
intersections of the hyperboloid with subspaces of $V$. The points at infinity in this
model correspond to isotropic one-dimensional subspaces of $V$.

A primitive vector $e$ of a quadratic lattice $L$ is
called a \textit{root} or, more precisely, a
$k$-\textit{root} if $(e,e)=k > 0$ and
$$
2(e, x) \in k \mathbb{Z} \quad \forall x \in L.
$$

Every root $e$ defines an orthogonal reflection (which is called
a $k$-\textit{reflection}) in the space $L \otimes \mathbb{R}$ by setting
$$
R_e \colon x \mapsto x - \dfrac{2(e, x)}{(e, e)} e,
$$
which preserves the lattice $L$. In the hyperbolic case, $R_e$ defines
a reflection on the hyperplane
$$
H_e =\{ x \in \mathbb{L}^n \colon (x,e) = 0\}
$$
in the space $\mathbb{L}^n$.

Then it is known that the group
$$
O '(L) = O (L) \cap O' (V)
$$
acts discretely on the Lobachevsky space and its fundamental polyhedron
has a finite volume.
Let $O_r (L)$ be the subgroup of $O '(L)$ generated by all
reflections contained in $O '(L)$. A hyperbolic lattice $L$ is called
\textit{reflective} if the subgroup $O_r (L)$  has a finite index in $O (L)$.
A hyperbolic lattice $L$ is called \textit{$1.2$-reflective}
if the subgroup $O^{(1.2)}_r (L)$ generated
by all $1$- and $2$-reflections  has a finite index in $O (L)$.

The lattice $L$ is reflective
if and only if the fundamental
polyhedron $M$ of the group $O_r (L)$ has a finite
volume in the Lobachevsky space $\mathbb{L}^n$.
The lattice $L$ is $1.2$-reflective
if and only if the fundamental
polyhedron $M ^ {(1.2)}$ of the group $O^{(1.2)}_r (L)$) has a finite
volume in $\mathbb{L}^n$.

There is an algorithm that, given a lattice $L$, enables
one to find recursively all faces of the polyhedron $M$ and determine
if there are only finitely many of them (\cite{vinberg73descrgr, vinberg72unitgr}).

A hyperbolic lattice $L$ is called \textit{isotropic} if the corresponding
quadratic form represents zero, otherwise it is called \textit{anisotropic}.

E.B. Vinberg (\cite{vinberg07classhyplat}) classified all $2$-reflective
hyperbolic lattices of rank $4$.
V.V.~Nikulin (\cite{nikulin81factor, nikulin84spacek3}) classified all
$2$-reflective hyperbolic lattices of
rank not equal to $4$, and  in \cite{nikulin2000hyp-root-sys}
he found all maximal reflective hyperbolic lattices of rank $3$.
Subsequently, D.~Allcock in his paper \cite{allcock2012refl-lat-3} classified
all reflective lattices of rank $3$.
In \cite{scharlauwalhorn92}, R.~Scharlau and C.~Walhorn presented  a
hypothetic list of all
maximal groups of the form $O_r (L)$, where $L$ is a reflective
isotropic hyperbolic lattice of rank $4$.

Our goal is to find all \textit{maximal} arithmetic groups generated
by $1$- and $2$-reflections. This means that we are looking for
$1.2$-reflective hyperbolic lattices $L$ of rank $4$ such that
the subgroup $O^{(1.2)}_r (L)$ is not contained in any other
arithmetic group generated by $1$- and $2$-reflections.
So we can assume that the lattice $L$ is \textit{maximal} (in particular,
its invariant factors are square-free), because
the group generated by $1$- and $2$-reflections can only increase by
passing to a superlattice ($1$- and $2$-roots always define a reflection).

In this paper, we confine ourselves to finding all maximal anisotropic
$1.2$-reflective hyperbolic lattices $L$ of rank $4$.

To state the main result of this paper we introduce some notation:

\begin{itemize}
\item  $[C]$ denotes  the quadratic lattice with the scalar product
given be a symmetric matrix $C$  in some basis,

\item $d(L)$ denotes  the discriminant of the lattice $L$,

\item $L \oplus M$ denotes the orthogonal sum of two lattices of $ L $ and $ M $,

\item $ [k]L $ denotes  the quadratic lattice obtained from  $L$ by multiplying
the scalar product by $k \in \mathbb{Z} $.
\end{itemize}

The main result of this paper is this.

\begin{theorem}\label{osn-res}
All maximal $1.2$-reflective anisotropic hyperbolic  lattices of rank $4$ are
presented in the following table:
\vskip 0.5 cm
\begin{small}
\begin{tabular} {| c | c | c | c | c |}
\hline
$L$ & Invariant factors & Discriminant \\
\hline
$[- 7] \oplus [1] \oplus [1] \oplus [1]$ & $(1, 1, 1, 7)$ &  $-7$ \\
\hline
$[- 15] \oplus [1] \oplus [1] \oplus [1]$ & $(1, 1, 1, 15)$ & $-15$ \\
\hline
\end{tabular}
\end{small}
\vskip 0.5 cm

These lattices are in fact  $2$-reflective, as proved in \cite{vinberg07classhyplat}.
\end{theorem}

The author expresses his deep gratitude to
E.B.~Vinberg for his advice, help and attention.

\section{The method of the outermost edge}

In this paper we employ a new method of the outermost edge,
which
is a modification of the method used by V.V.~Nikulin in his papers
\cite{nikulin81klassargr} and \cite{nikulin2000hyp-root-sys}.

In  \cite{nikulin81klassargr}, V.V.~Nikulin proved the following assertion.

\begin{theorem}\label{nik-th}
Let $\mathbb{L}^n$ be  the $n$-dimensional Lobachevsky space, let $M$ be an
acute-angled convex polyhedron in~$\mathbb{L}^n$, and let $e_0$ be a
fixed interior point of~$M$. If $F$ is the face
of the polyhedron $M$ of codimension $1$ that is outermost
from $e_0$, then,
for any faces $F_u$ and $F_v$ (of codimension $1$) of the polyhedron
$M$ adjacent to $F$   and having external normals $u$ and $v$, respectively, we have
$$
- (u, v) \le 14.
$$
\end{theorem}

Note that the number $-(u, v)$ is the cosine of the angle between the
faces $F_u$ and $F_v$ if they either intersect or are parallel. If they are
divergent, then
$$
-(u,v) = \cosh \rho (F_u, F_v),
$$
where $\rho (\cdot, \cdot)$ is the distance in the Lobachevsky space.

We shall go a somewhat other way.
Let $M \subset \mathbb{L}^3 $ be the fundamental polyhedron of the group
$O^{(1.2)}_r (L)$ for a maximal anisotropic $1.2$-reflective hyperbolic lattice $L$
(of rank $4$), let $e_0$ be a fixed point inside the polyhedron $M$, and let $E$ be
{\bf the outermost edge} from this point.

Let $e_1, e_2$ be normals to the two faces containing  $E$
and let $e_3, \ldots, e_{k+3}$ be normals to the framing faces, i.e.,
to the faces containing one of
the vertices of $E$, but not containing the edge $E$.

\begin{definition}
A vertex of an $n$-dimensional convex polyhedron is called simple
if it belongs to exactly $n$ faces.
\end{definition}

Note that the outermost edge $E$
connects two vertices, say, $V_1$ and
$V_2$, and each of them can be simple or non-simple,
and the number of  the framing faces
of $E$ changes
depending on it. Namely, $k = 1$ if both vertices of  $E$ are simple,
$k = 2$ if only one vertex is simple,  and $k = 3$ if both vertices are non-simple.
We consider only anisotropic lattices, so both vertices of $E$ must be simple,
that is, $k = 1$.

The following assertions show that the scalar product (taken with the
sign ``minus'') of the normals to the framing faces is bounded by some explicit number.
Note that we consider the case where  framing
faces are divergent, since otherwise the inner product taken
with the minus sign does not exceed one.

We have the following corollary of Theorem \ref{nik-th}.

\begin{proposition} \label{ud-rebro}
Let $e_0$ be a fixed point inside the acute-angled convex polyhedron
$M \subset \mathbb{L}^3$ and let $E$ be the outermost edge from it, and let $F$
be a face
containing this edge.
Let $E_1$, $E_2$ be disjoint edges of this face coming out from
different vertices of $E$. Then $ \cosh \rho (E_1, E_2) \le 14$.
\end{proposition}
\begin{proof}
Let $e_1$ be the projection of $e_0$ to the face $F$. Note that $e_1$
is an interior point of this face, since otherwise  the point
$e_0$ would lie outside of some dihedral angle adjacent to  $F$
(because the polyhedron $M$ is convex and acute-angled).
Further, since $E$ is
the outermost edge of the polyhedron for $e_0$, then
$$\rho (e_0, E) \ge \rho (e_0, E_i), \ i = 1, 2.$$
It follows from this and  the three
perpendiculars theorem  that the distance between the point $e_0$
and the edge $E$ is not
less than the distance between this point and all other edges of the face $F$.
This means that Theorem \ref{nik-th} is applicable to the point $e_1$
inside the polygon $F$.
\end{proof}

\begin{proposition}\label{sledstvie-rasst}
Let $F_u$ and $F_v$ be the faces of the polyhedron $M$ framing the
outermost edge $E$, passing through different vertices and containing the edges  $E_1$
 and $E_2$, respectively, mentioned in
Proposition \ref{ud-rebro}. Then $\cosh \rho (F_u, F_v) \le 14$ if  $F_u$
and $F_v$ diverge.
\end{proposition}
\begin{proof}
Clearly, the distance between the faces is not greater than the distance
between their edges.
Therefore,
$$
\cosh \rho (F_u, F_v) \le \cosh \rho (E_1, E_2) \le 14,
$$
as announced.
\end{proof}

Note that we are given bounds
on all elements of the matrix $G(e_1, e_2, e_3, e_4)$,
because all the faces $F_i$ are pairwise intersecting, excepting,
possibly, the pair of faces $F_3$ and $F_4$, but if they do not intersect,
then the distance between
these faces is bounded due to Proposition \ref{sledstvie-rasst}.
Thus, there are only finitely many possible matrices
$G(e_1, e_2, e_3, e_4)$.

The vectors $e_1, e_2, e_3, e_4$ generate some sublattice $L '$
of finite index in the lattice $L$. More precisely, the lattice $ L $ lies between
the lattices $L'$ and $(L')^*$, and
$$[(L ')^* \colon L']^2 = |d (L ')|,$$
where $d (L)$ denotes the discriminant of the lattice $L$.
Hence we have the inequality
$$[L \colon L ']^2 \le |d (L')|.$$
By using this estimate,  in each case generated by the lattice $L '$,
we find  all its extensions of finite index.

To reduce this sorting we use the fact that
$$
\det \ G (e_0, e_1, e_2, e_3, e_4) = 0.
$$

Note that $\det \ G(e_0, e_1, e_2, e_3, e_4)$ is a quadratic function
with respect to the element $(e_3, e_4)$. By using this observation and a number of
geometric considerations presented in Section \ref {osn-utv},
we shall obtain sharper bounds on the distance between the faces $F_3$ and $F_4$.
Thus, we shall obtain a finite list of maximal anisotropic
lattices that can be reflective.

Note also that if the lattice $L '$ is $1.2$-reflective, then the lattice
 $L$ containing
$L'$ is also $1.2$-reflective (passing to a superlattice, we add only
$1$- and $2$-reflections, which cut the fundamental polyhedron,
so that its volume remains finite).

So our goal is the implementation of
the following steps (for each step, we refer to the section below where it is made):

1) finding sharper bounds on the element $(e_3, e_4)$ of the matrix
$G (e_1, e_2, e_3, e_4)$ (Section~\ref{osn-utv});

2) finding all Gram matrices $G (e_1, e_2, e_3, e_4) $ and detecting
the type of the lattice $L'$ according to each matrix; next, picking  only anisotropic
lattices and finding all possible extensions (Section~\ref{sp-prom});

3) testing all maximal lattices on $1.2$-reflectivity
(Section~\ref{issled-refl}).

\section{Quadratic lattices}\label{quadratic-forms}

In this section we give some necessary information about
indefinite quadratic lattices. For more details, see \cite{kassels82qf} and
\cite{vinberg84abskr}.

Let $A$ be a principal ideal ring. A \textit{quadratic $ A $-module} is a
free $A$-module of finite rank equipped with a non-degenerate
 symmetric bilinear form
with values in $A$, called a \textit{scalar product}. In particular,
a quadratic $\mathbb{Z}$-module is called
a \textit{quadratic} lattice.
We denote by $[C]$ the standard module $A^n$ whose
scalar multiplication is defined by a Gram matrix $C$.

The determinant of the Gram matrix of a basis of a module $L$  is called
a \textit{discriminant} $d (L)$ of the quadratic $A$-module $L$.
It is defined up to a multiplication by an element of $(A^*)^2 $ ($ A ^ * $
denotes the group of invertible elements of the ring $A$) and can be regarded as
an element of the semigroup $ A /(A^*)^2$. Hence
$$
d ([C]) = \det C \cdot (A^*)^2.
$$
A quadratic $A$-module $L$ is called \textit{unimodular} if $d(L) \in A^*$.
In the case where $2 \not \in A^*$, the quadratic $A$-module $L$ is called
\textit{even} if $(x, x) \in 2A$ for any $x \in L$, and \textit{odd}
otherwise.

A nonzero vector $x \in L $ is called \textit{isotropic} if $(x, x) = 0$.
A quadratic
module $L$ is called \textit{isotropic} if it contains at least one
isotropic vector,
otherwise $L$ is called \textit{anisotropic}.

Since $(\mathbb{Z}^*)^2 = {1}$, the discriminant $d (L)$ of a
quadratic lattice $L$
is  an integer number. The unimodularity of a quadratic lattice $L$ is
equivalent to that $L$ coincides with its \textit {conjugate} lattice
$$
L^* = \{x \in L \otimes \mathbb{Q} \colon \forall y \in L \ \ (x, y)
\in \mathbb {Z} \}.
$$

For a lattice $L$, the invariant factors of the Gram matrix of a basis of $L$
are called \textit {invariant factors} of the lattice $L$.
The invariant factors of an integer matrix $G$ are defined through its
minors,
namely, if $D_k$ is
the greatest common divisor of all minors of order $k$, then
$E_k = \frac{D_k}{D_{k-1}}$ is an
invariant factor of the matrix $G$.
 It is known that
$ E_k \mid E_{k + 1} $ and the product of all invariant factors of the
lattice $L$ is equal to $|d(L)|$.

Every quadratic lattice $L$ defines a quadratic real vector
space
$L_{\infty} = L \otimes \mathbb{R}$ and, for any prime~$p$, it defines
a quadratic $\mathbb{O}_p$-module
$L_p = L \otimes \mathbb{O}_p$, where $\mathbb {O}_p$ is the ring
of $p$-adic numbers.
The signature of
 the lattice $ L $ is defined as the signature of the space $ L_{\infty}$.
 It is obvious that if two quadratic
lattices $L$ and $M$ are isomorphic, then
they have the same signature and $ L_p \simeq M_p $ for any prime $p$.
The converse is also true under the following conditions:

(i) $L$ is indefinite;

(ii) for any prime $ p $, the lattice $L$ has two invariant factors
divisible by  the same power of $p$.

The structure of quadratic $\mathbb{O}_p$-modules
can be described as follows.
Each such module
$L_p$ admits the \textit{Jordan decomposition}
$$
L_p = L_p^{(0)} \oplus [p] L_p^{(1)} \oplus [p^2] L_p^{(2)} \oplus \ldots,
$$
where all $L_p^{(j)}$ are unimodular quadratic $\mathbb{O}_p$-modules.
These unimodular modules are determined by $L$ uniquely up to an isomorphism, unless
$p \not = 2$. In case $p = 2$ the rank and the parity of each such module are uniquely
determined by $L$.

\begin{proposition}
If $L$ is a maximal quadratic lattice that is not contained
in any other quadratic lattice, then
$$
L_p = L_p^{(0)} \oplus [p] L_p^{(1)}
$$
for all primes $ p \ | \ d(L) $.
\end{proposition}
\begin{proof}
It is clear that if a lattice is maximal, then its invariant factors are
free from
the squares. Indeed, otherwise we can consider the lattice
$$
L'_p = L_p^{(0)} \oplus [p] L_p^{(1)} \oplus L_p^{(2)}
\oplus [p] L_p^{(3)} \oplus
\ldots.
$$
Then we have the following chain of embeddings:
$$
L_p \subset L'_p \subset L_p \otimes \mathbb{Q}_p,
$$
which hold due to the fact that the lattice $L'_p $ is derived from
the lattice $L_p$
by reducing some of the vectors by $p$. To complete the proof
it remains to apply the theorem on the existence of given quadratic
completions (see, e.g., Theorem 1.1 on p. 218 in \cite{kassels82qf}).
\end{proof}

\begin{definition}
Let $a, b \in \mathbb{Q}^*_ p $. Set

$(a, b)_p := 1$ if the equation $ax^2 + by^2 = 1$ has a solution in
$\mathbb{Q}^*_p$

and $(a, b)_p := -1 $ otherwise.

The number $(a, b)_p $ is called the Hilbert symbol.
\end{definition}

It is known that the group $ \mathbb{Q}^*_ p / (\mathbb{Q}^*_ p)^2$ can be
regarded as a vector space over $\mathbb{Z}_2 = \mathbb{Z} / 2 \mathbb{Z}$
of rank $2$  for $p \not = 2$
(respectively, of rank $3$ for $p = 2$). A
basis of this vector space is either
the set $ \{\varepsilon, p \}$ with $ p \not = 2 $, where $p$ is prime and $\varepsilon$ is a
quadratic non-residue modulo $p$, or the set $\{- 1, -3, p \}$ for $p = 2$.
The Hilbert symbol is a non-degenerate symmetric bilinear
form on this vector space. Its values on the basis elements
are well known.

Now we can define the Hasse invariant for an arbitrary quadratic
space $W$ over the field $\mathbb{Q}_p $. Let $f (x) $ be the quadratic form corresponding
to $W$ and let $a_1, \ldots, a_ {n}$ be its coefficients in the
canonical form.

\begin{definition}
The number
$$
\varepsilon_p (f) = \prod_{i <j} (a_i, a_j)_p
$$
is called the Hasse invariant of the quadratic form $f$.
\end{definition}

The following assertions are well-known
(see, e.g., \cite[Lemma 2.6, p.~76]{kassels82qf}).

\begin{theorem}\label{anisotr-lemma}
A quadratic $\mathbb{O}_p$-module $L$ of rank $4$
is anisotropic if and only if the following conditions hold:

(1) $ d(f) \in (\mathbb{Q}^*_ p)^2$;

(2) $\varepsilon_p (f) = - (- 1, -1)_p$.
\end{theorem}

\begin{theorem}
A quadratic lattice $L$ of rank $4$
is anisotropic if and only if it is
anisotropic over all fields $\mathbb{Q}_p $ including $p = \infty$.
\end{theorem}

\section{Auxiliary results}\label{osn-utv}

In this section we consider the method of the outermost edge in details. Let
$M$ be an acute-angled convex polyhedron in the three-dimensional
Lobachevsky space
$\mathbb{L}^3$, let $e_0$ be a fixed interior point of this
polyhedron, and let $E$ be the outermost edge from this point.
In this case $e_1$, $e_2$ are normals to the faces $F_1$ and $F_2$ containing the
 edge $E$, and $e_3$, $e_4$ are normals to the faces  $F_3$ and~$F_4$
passing through the vertices $V_1$ and $V_2$ of the edge $E$, respectively.

Let us consider the extended
Gram matrix
$$
G(e_0, e_1, e_2, e_3, e_4) =
\begin{pmatrix}
-1 & -x_1 & -x_2 & -x_3 & -x_4\\
-x_1 & d_1 & -\varepsilon_{12} & -\varepsilon_{13} & -\varepsilon_{14}\\
-x_2 & -\varepsilon_{12} & d_2 & -\varepsilon_{23} & -\varepsilon_{24}\\
-x_3 & -\varepsilon_{13} & -\varepsilon_{23} & d_3 & -T\\
-x_4 & -\varepsilon_{14} & -\varepsilon_{24} & -T & d_4
\end{pmatrix}
$$
where $d_i = (e_i, e_i)$ can equal $1$ or $2$, which corresponds to $1$- or
$2$-reflections. Then the numbers
$\varepsilon_{ij} = (e_i, e_j)$ can equal $0$ or $1$,  in addition,
 $$T = (e_3, e_4) \le 14 \sqrt{d_3 d_4},$$
 where $T$ is integer,
 and
$$x_j = -(e_0, e_j) = - \sinh \rho (e_0, H_j) \cdot \sqrt{(e_j, e_j)} > 0$$

Before starting investigation of Gram matrices we now derive a useful
formula for the distance from a point to a plane of arbitrary codimension
in the space $\mathbb{L}^n$.

\begin{theorem}\label{formula-rasst}
The distance from the point $e_0 \in \mathbb{L}^n$, where $(e_0, e_0) = -1$,
to the plane
$$H_{e_1, \ldots, e_k} := \{x \in \mathbb{L}^n \colon x \in \langle e_1,
\ldots, e_k \rangle^{\perp}\}$$
can be calculated by the formula
$$
\sinh^2 \rho(e_0,  H_{e_1, \ldots, e_k}) =\sum_{i, j}\overline{g_{ij}} y_i y_j,
$$
where $\overline{g_{ij}}$ are the elements of the inverse matrix
$G^{-1} = G(e_1, \ldots, e_k)^{-1}$, and $y_j = - (e_0, e_j)$ for all $1 \le j \le k$.
\end{theorem}
\begin{proof}
Let $f$ be the orthogonal projection $e_0$ to the plane
$H_{e_1, \ldots, e_k}$.
It is the intersection of the
plane $H_{e_1, \ldots, e_k}$ with the straight
line $\ell$ passing through the point $e_0$ and perpendicular to $H_{e_1, \ldots, e_k}$.
Since $\ell$ and $H_{e_1, \ldots, e_k}$ are orthogonal,  the intersection of
their defining subspaces
$\langle \ell \rangle$ and $\langle H_{e_1, \ldots, e_k} \rangle$
is a one-dimensional hyperbolic subspace~$\langle f' \rangle$. Hence the sections of
these subspaces by the subspace $\langle f' \rangle^{\perp}$
are orthogonal to each other.

It follows that
$$
\langle \ell \rangle = \langle f' \rangle \oplus \langle h' \rangle,
$$
where $h' \perp \langle H_{e_1, \ldots, e_k} \rangle$.

It remains to observe that the points $0, f, f '$ lie in the one-dimensional
hyperbolic subspace
$\langle f' \rangle$, hence $f = c f'$, where a positive
constant number $c$ can be found from the condition
$$(f, f) = c^2 (f', f') = -1.$$

Then the distance from the point $e_0 \in \mathbb{L}^n$, where $(e_0, e_0) = -1$,
to the plane
$H_{e_1, \ldots, e_k}$  equals the distance from this point
to its orthogonal projection, that is,
$$
\cosh \rho (e_0,  H_{e_1, \ldots, e_k}) = \cosh \rho (e_0, f) = - (e_0, f),
$$
whence we find that
$$
\cosh^2 \rho (e_0,  H_{e_1, \ldots, e_k}) = c^2 (e_0, f')^2
= -\dfrac{(e_0, f')^2}{(f', f')}
= -(f', f').
$$

Let us prove the following lemma.

\begin{lemma}
The following equation holds:
$$
\cosh \rho(e_0,  H_{e_1, \ldots, e_k}) = \sqrt{\dfrac{|\det
G(e_0, e_1, \ldots, e_k)|}{|\det G(e_1, \ldots,
e_k)|}}
$$
\end{lemma}
\begin{proof}
Indeed,
$$\det G(e_0, e_1, \ldots, e_k) = \det G(f', e_1, \ldots, e_k)
= (f', f') \det G(e_1, \ldots, e_k).
$$
Taking into account that $(f', f') < 0$, we have
$$
\dfrac{|\det G(e_0, e_1, \ldots, e_k)|}{|\det G(e_1, \ldots, e_k)|}
= -(f', f') =
\cosh^2 \rho (e_0, H_{e_1, \ldots, e_k}),
$$
which completes the proof of the lemma.
\end{proof}

It remains to show that
$$
-(f', f') = 1 + \overline{y}^T G^{-1} \overline{y},
$$
where
$G = G(e_1, \ldots, e_k)$, $\overline{y} = (y_1, \ldots, y_k)^T \in
\mathbb{R}^k$.
We observe that
$$
f' = e_0 + \lambda_1 e_1 + \ldots + \lambda_k e_k,
$$
and $(f', e_j) = 0$ for all $1 \le j \le k$. From these orthogonality conditions
we obtain that the column
$\overline{\lambda} = (\lambda_1, \ldots, \lambda_k)^T$ is
a solution to the system of linear equations  with the matrix $G$:
$$G \overline{\lambda} = \overline{y}.$$
Then $\overline{\lambda} = G^{-1}  \overline{y}$.
Therefore,
$$
(f', f') = (e_0, e_0) - 2(\overline{\lambda},  \overline{y})
+ \overline{\lambda}^T
G^{-1} \overline{\lambda} = -1 -  \overline{y}^T G^{-1}  \overline{y},
$$
whence we have that
$$
\sinh^2  \rho (e_0,  H_{e_1, \ldots, e_k}) = \cosh^2 \rho (e_0,
H_{e_1, \ldots, e_k}) - 1 =
-(f', f') - 1 = \overline{y}^T G^{-1}  \overline{y},
$$
as required.
\end{proof}

The fact that $E$ is the outermost edge from the point $e_0$
gives us the following estimates
on the elements of the matrix $G(e_0, e_1, e_2, e_3, e_4)$.

\begin{proposition}\label{oz-rasst-zel-matr-gr}
There are the following bounds on
the elements $x_j$ of the Gram matrix $G(e_0, e_1, e_2, e_3, e_4)$:
\begin{equation}\label{ur-1}
\dfrac{1}{d_1 d_2 - \varepsilon_{12}^2} (d_2 x_1^2
+ 2 \varepsilon_{12} x_1 x_2 + d_1 x_2^2)
\ge \dfrac{1}{d_1 d_3 - \varepsilon_{13}^2}  (d_3 x_1^2 +
2 \varepsilon_{13} x_1 x_3 + d_1 x_3^2)
\end{equation}
\begin{equation}\label{ur-2}
\dfrac{1}{d_1 d_2 - \varepsilon_{12}^2}(d_2 x_1^2
+ 2 \varepsilon_{12} x_1 x_2 + d_1 x_2^2)
 \ge \dfrac{1}{d_1 d_4 - \varepsilon_{14}^2}  (d_4 x_1^2 + 2
\varepsilon_{14} x_1 x_4 + d_1 x_4^2)
\end{equation}
\begin{equation}\label{ur-3}
\dfrac{1}{d_1 d_2 - \varepsilon_{12}^2}(d_2 x_1^2
+ 2 \varepsilon_{12} x_1 x_2 + d_1 x_2^2)
\ge  \dfrac{1}{d_2 d_3 - \varepsilon_{23}^2}  (d_3 x_2^2 + 2
\varepsilon_{23} x_2 x_3 + d_2 x_3^2)
\end{equation}
\begin{equation}\label{ur-4}
\dfrac{1}{d_1 d_2 - \varepsilon_{12}^2}(d_2 x_1^2
+ 2 \varepsilon_{12} x_1 x_2 + d_1 x_2^2)
 \ge \dfrac{1}{d_2 d_4 - \varepsilon_{24}^2}  (d_4 x_2^2 + 2
\varepsilon_{24} x_2 x_4 +  d_2 x_4^2)
\end{equation}
\end{proposition}
\begin{proof}
We observe that in inequalities to be proven we have
the hyperbolic sine of distance from $e_0$ to the edge $E$, and also
the hyperbolic sines of distances from $e_0$ to the other
edges passing through the vertices of $E$. 
Indeed, the straight line $\ell$ containing
the edge $E$ lies in the faces with normal vectors $e_1$
and $e_2$, i.e.,
$$
\ell  = H_{e_1, e_2}.
$$
Therefore,
$$
G(e_1, e_2)^{-1} =
\begin{pmatrix}
d_1 & -\varepsilon_{12}\\
-\varepsilon_{12} & d_2
\end{pmatrix}^{-1} = \dfrac{1}{\det G(e_1, e_2)}
\begin{pmatrix}
d_2 &  \varepsilon_{12}^2\\
\varepsilon_{12}^2 & d_1
\end{pmatrix}= \dfrac{1}{d_1 d_2 - \varepsilon_{12}^2}
\begin{pmatrix}
d_2 &  \varepsilon_{12}^2\\
\varepsilon_{12}^2 & d_1
\end{pmatrix}.
$$
and by Theorem \ref{formula-rasst} we have
$$
\sinh^2 \rho(e_0, E) = \sinh^2 \rho(e_0, \ell)= \sinh^2 \rho(e_0, H_{e_1, e_2}) =
\dfrac{1}{d_1 d_2 - \varepsilon_{12}^2}(d_2 x_1^2
+ 2 \varepsilon_{12} x_1 x_2 + d_1
x_2^2).
$$
Thus, if we denote by $E_{ij}$ the edge that is contained in
 the faces $F_i$ and $F_j$, then
inequalities (\ref{ur-1})--(\ref{ur-4}) take the following form:
$$
\sinh^2 \rho(e_0, E) \ge \sinh^2 \rho(e_0, E_{13}), \quad \sinh^2 \rho(e_0, E)
\ge \sinh^2
\rho(e_0, E_{23}),
$$
$$
\sinh^2 \rho(e_0, E) \ge \sinh^2 \rho(e_0, E_{14}), \quad \sinh^2 \rho(e_0, E) \ge
\sinh^2 \rho(e_0, E_{24}),
$$
and these inequalities are true due to the fact that $E$ is the outermost edge for the point $e_0$.
\end{proof}

As we have said before, the determinant of the extended matrix
 $G(e_0, e_1, e_2, e_3, e_4)$ vanishes, and due to this fact we can find
 sharper bounds on the number $T$.

\begin{lemma}\label{lem-nevozm}
In the above notation, let $\alpha_{ij}$ be the dihedral angle between
the faces
$F_i$ and $F_j$ in case they intersect.
If  the angle $\alpha_{12}$ is right,
then the case
$$\alpha_{13} = \alpha_{24} = \dfrac{\pi}{4}$$
is impossible (and also the similar case
$\alpha_{14} = \alpha_{23} = \dfrac{\pi}{4}$).
\end{lemma}
\begin{proof}
Since each number $d_j$ equals $1$ or $2$,  the indicated collection of angles can
appear precisely when $d_1\not = d_3$ and $d_2 \not = d_4$. The
following cases are possible up to renumbering of faces:

$
(i) \ d_1 = d_2 = 1,  \ d_3 = d_4 = 2, \  \varepsilon_{12} = 0,
\ \varepsilon_{13} =
\varepsilon_{24} = 1;
$

$
(ii) \ d_1 = d_4 = 1,  \ d_2 = d_3 = 2, \  \varepsilon_{12} = 0,
\ \varepsilon_{13} =
\varepsilon_{24} = 1;
$

$
(iii) \ d_1 = d_2 = 2,  \ d_3 = d_4 = 1, \  \varepsilon_{12} = 0,
\ \varepsilon_{13} =
\varepsilon_{24} = 1.
$

Then inequalities (\ref{ur-1}) and (\ref{ur-3}) from Proposition
\ref{oz-rasst-zel-matr-gr} take the following form in case~(i):
$$
x_1^2 + x_2^2 \ge 2x_1^2 + 2 x_1 x_3 + x_3^2, \quad
x_1^2 + x_2^2 \ge 2x_2^2 + 2 x_2 x_4 + x_4^2,
$$
whence we have
$$
x_2 \ge x_1 + x_3, \quad x_1 \ge x_2 + x_4,
$$
which is impossible due to the fact that all $x_j$ are  positive numbers.

The cases (ii) and (iii) are treated similarly.
\end{proof}

\begin{proposition}\label{per-sluch}
In the above notation, there are only
 the following cases up to renumbering of faces:

(1) $d_1 = d_2 = d_3 = d_4 = 1$,  and all numbers $\varepsilon_{ij} = 0$;

(2) $d_1 = d_2 = d_3 =  1$, $d_4 = 2$, $\varepsilon_{12}
= \varepsilon_{13} =\varepsilon_{23} = 0$,
$\varepsilon_{14}, \varepsilon_{24} \le 1$;

\hskip 0.5cm (2.0) all $\varepsilon_{ij} = 0$;

\hskip 0.5cm (2.1) $\varepsilon_{14} = 1$, $\varepsilon_{12}
= \varepsilon_{13} =\varepsilon_{23} =
\varepsilon_{24}= 0$;

(3) $d_1 = d_2 = 1$, $d_3 = d_4 = 2$, $\varepsilon_{12} = 0$,
all other $\varepsilon_{ij}$ can equal $0$ or $1$;

\hskip 0.5cm (3.0) all $\varepsilon_{ij} = 0$;

\hskip 0.5cm (3.1)  $\varepsilon_{13} = 1$, $\varepsilon_{12}
= \varepsilon_{14} =\varepsilon_{23} =
\varepsilon_{24}= 0$;

\hskip 0.5cm (3.2)  $\varepsilon_{13} = \varepsilon_{14} = 1$,
$\varepsilon_{12}  =\varepsilon_{23} =
\varepsilon_{24}= 0$;

(4) $d_1 = 1$, $d_2 =  2$,  $d_3 = d_4 = 1$, $\varepsilon_{13}
= \varepsilon_{14} =
0$, $\varepsilon_{12}, \varepsilon_{23},  \varepsilon_{24} \le 1$;

\hskip 0.5cm (4.0) all $\varepsilon_{ij} = 0$;

\hskip 0.5cm (4.1) precisely one number  $\varepsilon_{ij}$ equals $1$;

\hskip 1cm (4.1.1)  $\varepsilon_{12} = 1$;

\hskip 1cm (4.1.2)  $\varepsilon_{23} = 1$;

\hskip 0.5cm (4.2) precisely two numbers  $\varepsilon_{ij}$ equal $1$;

\hskip 1cm (4.2.1)  $\varepsilon_{23} = \varepsilon_{24} = 1$;

(5) $d_1 = 1$, $d_2 = d_3 = 2$,  $d_4 = 1$, $\varepsilon_{14} = 0$,
all other $\varepsilon_{ij}$ can equal  $0$ or $1$;

\hskip 0.5cm (5.0) all $\varepsilon_{ij} = 0$;

\hskip 0.5cm (5.1) precisely one number  $\varepsilon_{ij}$ equals $1$;

\hskip 1cm (5.1.1)  $\varepsilon_{12} = 1$;

\hskip 1cm (5.1.2)  $\varepsilon_{13} = 1$;

\hskip 1cm (5.1.3)  $\varepsilon_{23} = 1$;

\hskip 0.5cm (5.2) precisely two numbers $\varepsilon_{ij}$ equal $1$;

\hskip 1cm (5.2.1)  $\varepsilon_{12} = \varepsilon_{23} = 1$;

\hskip 1cm (5.2.2)  $\varepsilon_{13} = \varepsilon_{23} = 1$;

\hskip 1cm (5.2.3)  $\varepsilon_{23} = \varepsilon_{24} = 1$;

(6) $d_1 = 1$, $d_2 =  d_3 = d_4 = 2$, $\varepsilon_{ij}$ can equal $0$ or $1$;

\hskip 0.5cm (6.0) all $\varepsilon_{ij} = 0$;

\hskip 0.5cm (6.1) precisely one number  $\varepsilon_{ij}$  equals $1$;

\hskip 1cm (6.1.1)  $\varepsilon_{12} = 1$;

\hskip 1cm (6.1.2)  $\varepsilon_{13} = 1$;

\hskip 1cm (6.1.3)  $\varepsilon_{23} = 1$;

\hskip 0.5cm (6.2) precisely two numbers  $\varepsilon_{ij}$  equal $1$;

\hskip 1cm (6.2.1)  $\varepsilon_{12} = \varepsilon_{23} = 1$;

\hskip 1cm (6.2.2)  $\varepsilon_{13} = \varepsilon_{23} = 1$;

\hskip 1cm (6.2.3)  $\varepsilon_{13} = \varepsilon_{14} = 1$;

\hskip 1cm (6.2.4)  $\varepsilon_{13} = \varepsilon_{24} = 1$;

\hskip 1cm (6.2.5)  $\varepsilon_{23} = \varepsilon_{24} = 1$;

\hskip 0.5cm (6.3) precisely three numbers $\varepsilon_{ij}$ equal $1$,
all other ones equal $0$;

\hskip 1cm (6.3.1)  $\varepsilon_{13} = \varepsilon_{14} = 0$;

\hskip 1cm (6.3.2)  $\varepsilon_{12} = \varepsilon_{14} = 0$;

\hskip 1cm (6.3.3)  $\varepsilon_{12} = \varepsilon_{24} = 0$;

\hskip 0.5cm (6.4) precisely four numbers $\varepsilon_{ij}$ equal $1$,
the last number  equals $0$;

\hskip 1cm (6.4.1)  $\varepsilon_{12} = 0$;

(7) $d_1 = d_2 = 2$, $d_3 = d_4 = 1$,  $\varepsilon_{ij}$ can equal $0$ or $1$;

\hskip 0.5cm (7.0) all $\varepsilon_{ij} = 0$;

\hskip 0.5cm (7.1) precisely one number $\varepsilon_{ij}$  equals $1$;

\hskip 1cm (7.1.1)  $\varepsilon_{12} = 1$;

\hskip 1cm (7.1.2)  $\varepsilon_{13} = 1$;

\hskip 0.5cm (7.2) precisely two numbers  $\varepsilon_{ij}$ equal $1$;

\hskip 1cm (7.2.1)  $\varepsilon_{12} = \varepsilon_{13} = 1$;

\hskip 1cm (7.2.2)  $\varepsilon_{13} = \varepsilon_{14} = 1$;

\hskip 0.5cm (7.3) precisely three numbers $\varepsilon_{ij}$ equal $1$, all other
ones  equal  $0$;

\hskip 1cm (7.3.1)  $\varepsilon_{23} = \varepsilon_{24} = 0$;

\hskip 1cm (7.3.2)  $\varepsilon_{13} = \varepsilon_{24} = 0$;

(8) $d_1 = d_2 = 2$, $d_3 = 2$, $d_4 = 1$, $\varepsilon_{ij}$ can  equal  $0$ or $1$;

\hskip 0.5cm (8.0) all $\varepsilon_{ij} = 0$;

\hskip 0.5cm (8.1) precisely one number  $\varepsilon_{ij}$ equals $1$;

\hskip 1cm (8.1.1)  $\varepsilon_{12} = 1$;

\hskip 1cm (8.1.2)  $\varepsilon_{23} = 1$;

\hskip 1cm (8.1.3)  $\varepsilon_{14} = 1$;

\hskip 0.5cm (8.2) precisely two numbers $\varepsilon_{ij}$  equal  $1$;

\hskip 1cm (8.2.1)  $\varepsilon_{12} = \varepsilon_{13} = 1$;

\hskip 1cm (8.2.2)  $\varepsilon_{12} = \varepsilon_{14} = 1$;

\hskip 1cm (8.2.3)  $\varepsilon_{13} = \varepsilon_{14} = 1$;

\hskip 1cm (8.2.4)  $\varepsilon_{14} = \varepsilon_{23} = 1$;

\hskip 1cm (8.2.5)  $\varepsilon_{13} = \varepsilon_{23} = 1$;

\hskip 0.5cm (8.3) precisely three numbers  $\varepsilon_{ij}$  equal $1$,
all other ones equal $0$;

\hskip 1cm (8.3.1)  $\varepsilon_{23} = \varepsilon_{24} = 0$;

\hskip 1cm (8.3.2)  $\varepsilon_{23} = \varepsilon_{14} = 0$;

\hskip 1cm (8.3.3)  $\varepsilon_{12} = \varepsilon_{14} = 0$;

(9) $d_1 = d_2 = d_3 = d_4 = 2$, $\varepsilon_{ij}$ can  equal $0$ or $1$.

\hskip 0.5cm (9.0) all $\varepsilon_{ij} = 0$;

\hskip 0.5cm (9.1) precisely one number $\varepsilon_{ij}$ equals $1$;

\hskip 1cm (9.1.1)  $\varepsilon_{12} = 1$;

\hskip 1cm (9.1.2)  $\varepsilon_{23} = 1$;

\hskip 0.5cm (9.2) precisely two numbers  $\varepsilon_{ij}$ equal $1$;

\hskip 1cm (9.2.1)  $\varepsilon_{12} = \varepsilon_{13} = 1$;

\hskip 1cm (9.2.2)  $\varepsilon_{13} = \varepsilon_{23} = 1$;

\hskip 1cm (9.2.3)  $\varepsilon_{13} = \varepsilon_{14} = 1$;

\hskip 1cm (9.2.4)  $\varepsilon_{13} = \varepsilon_{24} = 1$;

\hskip 0.5cm (9.3) precisely three numbers $\varepsilon_{ij}$ equal $1$,
all other ones equal $0$;

\hskip 1cm (9.3.1)  $\varepsilon_{12} = \varepsilon_{13} = 0$;

\hskip 1cm (9.3.2)  $\varepsilon_{13} = \varepsilon_{14} = 0$;

\hskip 1cm (9.3.3)  $\varepsilon_{13} = \varepsilon_{24} = 1$;

\hskip 0.5cm (9.4) precisely four numbers of $\varepsilon_{ij}$ equal $1$,
the last number equals $0$;

\hskip 1cm (9.4.1)  $\varepsilon_{12} = 0$.

\end{proposition}
\begin{proof}
First of all we observe
that we have a variety of options for the location of units and twos on the
diagonal of the Gram matrix $G(e_0, e_1, e_2, e_3, e_4)$. The cases where
 all $d_j = 1$ or
all $d_j = 2$ are covered by general assertions (1) and (9) of this proposition.
If precisely one number of $d_j$ equals  $2$, then, up to renumbering
 of faces, we distinguish  only two cases: the first, where
  one of the faces $F_1$ or $F_2$
corresponds to a $2$-reflection, and the second, where one of the framing faces
($F_3$ or $F_4$) corresponds to this reflection.
These cases are covered by
(2) and (4). Similarly we consider the two cases where
precisely one number  $d_j$ equals $1$, which is covered by
(6) and (8). Finally, we have the cases where
precisely two numbers among $d_j$  equal $1$ and the other two equal $2$.
These cases are distinguished as follows: when $2$-reflections correspond
to the faces $F_1$ and $F_2$ (case~(7)), when one of  $2$-reflections
corresponds to the face containing
the edge $E$ and the second $2$-reflection
 corresponds to one of the framing faces (case~(5)),
and when both $2$-reflections correspond
to the framing faces (case~(3)).

In  case (1) all matrix elements (excepting the number $T$) are uniquely
determined. In all other
cases some of the numbers $\varepsilon_{ij}$ are uniquely determined,
and some are not, but in every
case all matrix elements satisfy the inequalities from
Proposition~\ref{oz-rasst-zel-matr-gr}.

We now observe that each of the vertices  $V_1$ and $V_2$ is characterized  by
the set of faces containing it
and the dihedral angles between them, i.e., for one of these vertices  we have a collection of
numbers
$(d_1, d_2, d_3; \varepsilon_{12}, \varepsilon_{13}, \varepsilon_{23})$,
and for the second one we have
$(d_1, d_2, d_4; \varepsilon_{12}, \varepsilon_{14}, \varepsilon_{24})$.

\begin{lemma}\label{vertices}
The outermost edge $E$ cannot have vertices of the following types:

$(1, 1, 2; 0, 1, 1)$, $(2, 1, 1; 1, 1, 0)$, $(1, 2, 1; 1, 0, 1)$,
$(2, 2, 1; 0, 1, 1)$, $(2, 1, 2; 1, 0, 1)$,\\
$(1, 2, 2; 1, 1, 0)$,
$(2, 2, 1; 1, 1, 1)$, $(2, 1, 2; 1, 1, 1)$, $(1, 2, 2; 1, 1, 1)$,
$(2, 2, 2; 1, 1, 1)$.
\end{lemma}
\begin{proof}
Indeed, for the first six collections, the sums of the dihedral angles at the corresponding vertices  equal
$$
\dfrac{\pi}{2}+\dfrac{\pi}{4}+\dfrac{\pi}{4} = \pi,
$$
for the next three collections these sums equal
$$
\dfrac{\pi}{3}+\dfrac{\pi}{4}+\dfrac{\pi}{4} < \pi,
$$
and for the last collection such a sum  equals
$$
\dfrac{\pi}{3}+\dfrac{\pi}{3}+\dfrac{\pi}{3} = \pi.
$$
Thus,
for all these collections we have that the sum of the dihedral angles at the simple vertex
in the three-dimensional Lobachevsky space is less than or equal to $\pi$, but it must be strictly larger than $\pi$.
\end{proof}

In case (2) only the numbers $\varepsilon_{14}$ and
$\varepsilon_{24}$ can  equal $1$. When both are zero, we obtain case
(2.0).
When only one of them equals $1$, we obtain two
identical cases up to renumbering, and they are covered by case (2.1).
But these numbers cannot  equal $1$ simultaneously by  Lemma \ref{vertices}.

In case (3) we  obtain immediately that all cases with precisely one unit
are symmetric to each other, i.e., it suffices to consider case (3.1).
If we have precisely two units in the collection of numbers $\varepsilon_{ij}$, then by symmetry and Lemmas
\ref{vertices} and \ref{lem-nevozm} it remains to consider only case (3.2).
It follows from the same lemmas that the case with three units in the collection of numbers
$\varepsilon_{ij}$ is impossible.

The remaining cases are considered similarly. Some options can be
 omitted, because
they differ one from another only by renumbering of faces, and also
there are some impossible cases, which are described in Lemmas \ref{vertices}
and \ref{lem-nevozm}.
\end{proof}

\begin{proposition}\label{sluch-perp}
Let all intersecting faces $F_j$ be pairwise perpendicular,
i.e., $\varepsilon_{ij} = 0$ for all $i,j$.
Then only the following two options for the number $T$ are possible:

(1) $T=3$ if $d_3 = d_4 = 2$;

(2) $T = 2$ if $d_3 = 2$, $d_4 = 1$ or $d_3 = 1$, $d_4 = 2$.

It follows from this that case (1) of Proposition \ref{per-sluch} is impossible.
\end{proposition}
\begin{proof}
Reducing all elements of the matrix if  needed (namely,
we can reduce the first column and row by $\sqrt{d_1}$,
the second  ones by $\sqrt{d_2}$, the third ones by $\sqrt{d_3}$, and the fourth  by
$\sqrt{d_4}$), we can assume in this case that the Gram matrix
 has the form
$$
G(e_0, e_1, e_2, e_3, e_4) =
\begin{pmatrix}
-1 & -y_1 & -y_2 & -y_3 & -y_4\\
-y_1 & 1 & 0 & 0 & 0\\
-y_2 & 0 & 1 & 0 & 0\\
-y_3 & 0  & 0 & 1 & -f\\
-y_4 & 0 & 0 & -f & 1
\end{pmatrix},
$$
i.e.,
$$
y_j = \dfrac{x_j}{\sqrt{d_j}}, \quad f = \dfrac{T}{\sqrt{d_3 d_4}},
$$
moreover, we can assume that $y_3 \le y_4$ and that the  inequalities from Proposition \ref{oz-rasst-zel-matr-gr} have the following form:
$$
y_3 \le y_4 \le y_1, y_2.
$$
Then
$$
\det G(e_0, e_1, e_2, e_3, e_4) = (y_1^2 + y_2^2 + 1)f^2 - 2 y_3 y_4 f
- (y_1^2 + y_2^2 + y_3^2 +
y_4^2 + 1) = 0,
$$
i.e.,
$$
f = \dfrac{y_3 y_4 + \sqrt{y_3^2 y_4^2 + (y_1^2 + y_2^2 + 1)(y_1^2 + y_2^2
+ y_3^2 + y_4^2
+ 1)}}{y_1^2 + y_2^2 + 1} \le A + \sqrt{A^2 + B + 1},
$$
where
$$
A:= \dfrac{y_3 y_4}{y_1^2 + y_2^2 + 1}, \quad B:=  \dfrac{y_3^2
+ y_4^2}{y_1^2 + y_2^2 + 1}.
$$
Therefore,
$$
2A \le B  \le \dfrac{y_1^2 + y_2^2}{y_1^2 + y_2^2 + 1} < 1,
$$
i.e.,
$$
f < 0.5 + \sqrt{0.25 + 1 + 1} = 2
$$
It remains to observe that $f  > 1$, since the faces
$F_3$ and $F_4$ diverge in this case, whence we have
$$
1 < \dfrac{T}{\sqrt{d_3 d_4}} < 2.
$$
Hence the integer number $T$ lies in the interval $(1; 4)$,
i.e., it can equal only  $2$ or $3$.
If $T=2$, then $1 \le \sqrt{d_3 d_4} < 2 < 2\sqrt{d_3 d_4}
 \le 4$, whence we obtain that one
of the numbers $d_3$  or $d_4$ equals $2$ and the second one equals $1$,
and this
brings us to case (2) of this theorem. Similarly, if $T=3$, then $d_3 = d_4 = 2$,
otherwise $\dfrac{T}{\sqrt{d_3 d_4}} > 2$.
\end{proof}

Let us  formulate the final result for the number $T$.

\begin{theorem}\label{ozenki-T-rasst}
In the above notation, suppose that for the number $T$ we have

$T = - \cosh \rho (F_3, F_4) \sqrt{d_3 d_4}$ if the faces $F_3$ and $F_4$ diverge,

$T = -\sqrt{d_3 d_4}$ if these faces are parallel,

$T = - \cos \angle(F_3, F_4) \sqrt{d_3 d_4}$ if they intersect.
Then the bounds on the number $T$ in all cases of Proposition \ref{per-sluch} are
given in the following table:

\vskip 0.5cm
\begin{footnotesize}
\begin{tabular}{|c|c|*{10}{c|}|c|}
\hline
\multirow{2}*{\#1} & \multicolumn{8}{|c|}{Points of the
 Proposition \ref{per-sluch}} \\ \cline{2-9}
 & (2) & (3) & (4) & (5) & (6) & (7) & (8) & (9)\\
\hline
\hline
0 &   $T=2$ &  $T=3$ &    &  $T=2$ &  $T=3$ &   &  $T=2$ &  $T=3$ \\
\hline
\hline
\multirow{3}*{1} &   $\varepsilon_{14} = 1$ &  $\varepsilon_{13} = 1$
&  $\varepsilon_{12} = 1$ &  $
 \varepsilon_{12} = 1$ &  $\varepsilon_{12} = 1$ &  $\varepsilon_{12} = 1$
&  $\varepsilon_{12} = 1$ &
 $\varepsilon_{12} = 1$ \\
&   $T<2$ & $T<3$ & $T<3$ & $T<5$ & $T<4$ & $T<3$ &  $T<4$ & $T<5$
 \\ \cline{2-9}
&    &  &  $\varepsilon_{23} = 1$ & $\varepsilon_{13} = 1$
& $\varepsilon_{13}
 = 1$ &
$\varepsilon_{13} = 1$ & $\varepsilon_{23} = 1$ & $\varepsilon_{23} = 1$
\\
&   &  &  $T<3$ & $T<2$ & $T<4$ & $T<2$ & $T<7$ & $T<5$ \\ \cline{4-9}
&   &  &  & $\varepsilon_{23} = 1$ &  $\varepsilon_{23} = 1$
&  & $\varepsilon_{14} = 1$ & \\
&  &  &  & $T<3$ & $T<4$ &  & $T<4$ &  \\ \cline{6-8}
\hline
\hline
\multirow{3}*{2} &    &  $\varepsilon_{13} =\varepsilon_{14} = 1$ &
$\varepsilon_{23} =\varepsilon_{24} = 1$ &   $\varepsilon_{12}
=\varepsilon_{23} = 1$ &  $\varepsilon_{12} =\varepsilon_{23} = 1$
&   $\varepsilon_{12} =\varepsilon_{13} = 1$ &  $\varepsilon_{12}
=\varepsilon_{13} = 1$ &   $\varepsilon_{12} =\varepsilon_{13} = 1$ \\
&  & $T<6$ & $T<3$ & $T<5$ & $T<8$ & $T<6$ &  $T<5$ & $T<6$
\\ \cline{3-9}
&  &  &  &  $\varepsilon_{13} =\varepsilon_{23} = 1$ &
$\varepsilon_{13} =\varepsilon_{23} = 1$ & $\varepsilon_{13}
=\varepsilon_{14} = 1$ &
$\varepsilon_{12} =\varepsilon_{14} = 1$ & $\varepsilon_{13}
=\varepsilon_{23} = 1$  \\
  &  &  &   & $T<5$ & $T<5$ & $T<3$ & $T<7$ & $T<4$ \\ \cline{5-9}
  &  &  &  & $\varepsilon_{23} =\varepsilon_{24} = 1$
& $\varepsilon_{13} =\varepsilon_{14} = 1$ &  & $\varepsilon_{13}
=\varepsilon_{14} = 1$ &
  $\varepsilon_{13} =\varepsilon_{14} = 1$\\
  &  &  &  & $T<4$  & $T<6$ &  & $T<7$ & $T<5$ \\ \cline{5-9}
  &  &  &  &  & $\varepsilon_{13} =\varepsilon_{24} = 1$& &
$\varepsilon_{14} =\varepsilon_{23} = 1$   &$\varepsilon_{13}
=\varepsilon_{24} = 1$\\
  &  &  &  & & $T<4$ &  & $T<7$ &  $T<4$ \\ \cline{6-9}
&  &  &  &  & $\varepsilon_{23} =\varepsilon_{24} = 1$& &
$\varepsilon_{13} =\varepsilon_{23} = 1$& \\
 &  &  &  &  & $T<6$ &  &  $T<4$ &  \\ \cline{6-9}
\hline
\hline
\multirow{3}*{3}  &  &   &  &    &  $\varepsilon_{13}
=\varepsilon_{14} = 0$ &  $\varepsilon_{23} =\varepsilon_{24} = 0$
&  $\varepsilon_{23} =\varepsilon_{24} = 0$ &  $\varepsilon_{12}
=\varepsilon_{13} = 0$  \\
& &  &  & & $T<8$ & $T<8$ &  $T<7$ & $T<6$ \\ \cline{6-9}
&  &  &  &  &  $\varepsilon_{12} =\varepsilon_{14} = 0$ &
$\varepsilon_{13} =\varepsilon_{24} = 0$
& $\varepsilon_{23} =\varepsilon_{14} = 0$ & $\varepsilon_{13}
=\varepsilon_{14} = 0$ \\
 &  &  &   &  & $T<7$ & $T<5$ & $T<7$ & $T<7$ \\ \cline{6-9}
 &  &  &  &  & $\varepsilon_{12} =\varepsilon_{24} = 0$ &
& $\varepsilon_{12} =\varepsilon_{14} = 0$
 &   $\varepsilon_{13} =\varepsilon_{24} = 0$\\
  &  &  &  &  & $T<6$ &  & $T<7$ &  $T<7$\\ \cline{6-9}
\hline
\hline
\multirow{2}*{4}   &  &  &  &   & $\varepsilon_{12} = 0$ &  &
&  $\varepsilon_{12} = 0$ \\
&  &  &  & & $T<6$  &  &  &  $T<7$  \\ \cline{6-9}
\hline
\end{tabular}
\end{footnotesize}
\vskip 0.5cm

This table contains estimates for each case of Proposition \ref{per-sluch}.
To each cell of this table, we  associate the number of units in the collection of
numbers $\varepsilon_{ij}$ (the number of units is denoted by the symbol
$\# 1$) and also one of general cases (2)---(9).
\end{theorem}
\begin{proof}
Let us present here only the main ideas of the proof  of this theorem.
It is readily seen that the proof reduces to investigation of cases of
Proposition \ref{per-sluch}. In each of this cases, we consider the determinant of the
Gram  matrix $G(e_0, e_1, e_2, e_3, e_4)$ similarly to the proof of
Proposition \ref{sluch-perp}.

In all these cases this determinant is a quadratic function with respect to the
integer valued variable $T$ with real parameters $x_1, x_2, x_3, x_4$:
$$
\det G(e_0, e_1, e_2, e_3, e_4)
= a(x_1, x_2, x_3, x_4) T^2 - 2b(x_1, x_2, x_3, x_4) T +
c(x_1, x_2, x_3, x_4)
= 0.
$$
For each case of Proposition \ref{per-sluch}
one can write some explicit bounds on the numbers
$x_j$ from Proposition \ref{oz-rasst-zel-matr-gr}
and then solve the quadratic equation
$$
T = \dfrac{b(x_1, x_2, x_3, x_4) + \sqrt{b(x_1, x_2, x_3, x_4)^2
+ a(x_1, x_2, x_3, x_4)c(x_1, x_2, x_3,
x_4)}}{a(x_1, x_2, x_3, x_4)} =
$$
$$
= A(x_1, x_2, x_3, x_4)
+ \sqrt{A^2(x_1, x_2, x_3, x_4) + B(x_1, x_2, x_3, x_4)},
$$
where
$$
A(x_1, x_2, x_3, x_4):
= \dfrac{b(x_1, x_2, x_3, x_4)}{a(x_1, x_2, x_3, x_4)}, \quad
B(x_1, x_2, x_3, x_4):
= \dfrac{c(x_1, x_2, x_3, x_4)}{a(x_1, x_2, x_3, x_4)}.
$$

It remains to bound the rational functions~$A(x_1, x_2, x_3, x_4)$ and~$B(x_1,
x_2, x_3, x_4)$ with the aid of the known estimates on $x_j$.
As a result we obtain some restrictions on the number $T$. For each of these
 subcases  we obtain the corresponding table.
\end{proof}

\section{Methods of verification of $1.2$-reflectivity of lattices}
\label{metod-issled-refl}

Here we describe some methods of verification
 of $1.2$-reflectivity of lattices.

\subsection{Vinberg's algorithm}

There is an affective algorithm of constructing
the fundamental polyhedron $M$
for groups generated by reflections.
We choose a basic point $v_0 \in \mathbb{E}^{n,1}$.
Denote by $\Gamma$ the group $O'(L)$
of integer linear transformations preserving the quadratic lattice $L$ and
not transposing
the future light cone and the past light cone.
Let $M_0$ be the fundamental polyhedral cone
for the group $(\Gamma_{v_o})_r = (\Gamma_r)_{v_0}$. Let
$H_1, \ldots , H_m$ be the sides of this cone and let $a_1, \ldots , a_m$ be the
corresponding outer normals. Then we can define the half-spaces
$$
H_k^- = \{x \in \mathbb{E}^{n,1}\colon (x, a_k) \le 0 \},
$$
in addition, we can define in the same way the half-space $H^-$
for every hyperplane $H$. Then we observe that the fundamental polyhedral cone
is the intersection of the half-spaces of this cone.

There is the unique camera of the group  $\Gamma_r$ contained in $M_0$ and
containing the point $v_0$. Inductively we can find the sides $H_{m+1}, \ldots$ of the
polyhedron $M$  and the corresponding outer normals $a_{m+1}, \ldots$.
Namely, at the $k$th step we pick a mirror $H_k$
and a vector $a_k$ orthogonal to it  such that

1) $(a_k, v_0) < 0$;

2) $(a_k, a_i) \le 0$ for all $i < k$;

3) the distance $\rho(v_0, H_k)$ is minimal under the conditions 1) and 2).

There is the following useful result (see
 Proposition~24 in~\cite{vinberg84abskr}).

\begin{proposition}
A quadratic lattice $L$ can have $k$-roots if and only if the
doubled largest invariant factor of the lattice $L$ is divisible by $k$.
\end{proposition}

\begin{theorem} (\text{E.B. Vinberg} \cite{vinberg73descrgr})
The polyhedron  $M$ can be found in the following way:
$$
M = \bigcap \limits_k H_k^- = \{x \in \mathbb{E}^{n,1}\colon (x, a_i)
\le 0 \ \hbox{for all} \ i\},
$$
and, in addition, all $H_k$ are the sides of $M$.
\end{theorem}

To each vertex of the Coxeter polyhedron there corresponds either
an elliptic subdiagram
of rank~$3$
of the Coxeter diagram
(a simple vertex) or a parabolic subdiagram of rank $2$
(a vertex at infinity).
The polyhedron  has a finite volume  if and only if it has at least
one vertex and each edge coming out from its vertex ends in another vertex.
Thus,  the Coxeter diagram enables us to determine whether the polyhedron
has a finite volume.

\subsection{The method of ``bad'' reflections}

There are some cases where a lattice is reflective, but it is not so easy
to determine its $1.2$-reflectivity.
One can consider the group $\Delta$ generated by the ``bad'' reflections
(i.e., which are not $1$- and $2$-reflections)
 in the sides of the fundamental polyhedron of the group $O_r (L)$.
The following lemma holds (see \cite{vinberg07classhyplat}).

\begin{lemma}
A lattice $L$ is $1.2$-reflective if and only if
it is reflective and the group
$\Delta$ is finite.
\end{lemma}

If we can construct the fundamental polyhedron of the group $O_r (L)$
for some reflective
lattice $L$, then we can find the faces in the Coxeter diagram
 not corresponding to $1$- and
$2$-reflections. If the group corresponding to the selected faces is finite,
then this lattice
is  $1.2$-reflective,
and if this group $\Delta$ is infinite,
 then this lattice is reflective, but not $1.2$-reflective.

\subsection{Passage to a smaller dimension}

There is an important theorem
that enables us to check the lattice reflectivity or non-reflectivity
via reduction of dimension.

\begin{theorem}\label{bugaenko}
(V.O. Bugaenko, see \cite[Theorem 2.1]{bugaenko92arcrgr-refl})

Let a hyperbolic lattice $L$ be decomposed into
the direct sum of a hyperbolic lattice $L'$ and
an elliptic lattice $M$. Let also $\Lambda$ and $\Lambda'$ be the Lobachevsky spaces
whose models are constructed in the spaces $V
= L \otimes \mathbb{R}$ and $V' = L' \otimes
\mathbb{R}$, respectively. Then the intersection of the fundamental
polyhedron of the group $O_r(L)$ with the subspace $\Lambda'$ of the
space $\Lambda$ is either empty or
is the fundamental polyhedron $P'$ of the group $O_r (L')$.
\end{theorem}

\begin{corollary}
If the lattice $L$ is reflective, then $L'$ is reflective as well.
\end{corollary}

Thus, we have a possibility to prove non-reflectivity of some lattices
by using their
decompositions into direct sums of a non-reflective lattice and an elliptic lattice.
To this end, we make  use of some results due to  Nikulin, namely,
if some maximal lattice $L$ of rank $3$
is not in his list of maximal reflective lattices of rank $3$, then
it follows that, for example, the  lattice $L \oplus [1]$
is  also not reflective.

\section{The list of intermediate lattices}\label{sp-prom}

\subsection{Intermediate lattices and their extensions}

%\vskip 1cm

Due to Proposition \ref{per-sluch}, Theorem \ref{anisotr-lemma},
and Theorem \ref{ozenki-T-rasst} we obtain a possibility
to create a programme that picks
from a finite collection of Gram matrices  those matrices that
can correspond to only anisotropic lattices. Our program has been
built in the computer algebra package Sage. As a result, we have
obtained in output the
matrices $G_1$---$G_{27}$, for each of which we obtain the
corresponding maximal extension.

The maximal anisotropic lattices  arising in this process  will be denoted
consecutively by  $L(k)$.

\begin{small} 

\vskip 1cm

(3)

\vskip 1cm

$
G_{1} =
\begin{pmatrix}
1 & 0 & -1 & -1\\
0 & 1 & 0 & 0\\
-1 & 0 & 2 & -3\\
-1 & 0 & -3 & 2
\end{pmatrix} \sim
\begin{pmatrix}
1 & 0 & 0 & 0\\
0 & 1 & 0 & 0\\
0 & 0 & 1 & -4\\
0 & 0 & -4 & 1
\end{pmatrix}, \ \det G_1 = -15, \\
$
%\vskip 0.4cm
hence we see that $L'_1 \simeq [-15] \oplus [1] \oplus [1] \oplus [1]$,
this is a maximal anisotropic lattice, we denote it by $L(1)$.
This lattice appears in \cite{vinberg07classhyplat} as an anisotropic
$2$-reflective hyperbolic lattice of rank $4$.

\vskip 1cm

(4)

\vskip 1cm

$
G_{2} =~
\begin{pmatrix}
1 & 0 & 0 & 0\\
0 & 2 & -1 & 0\\
0 & -1 & 1 & -2\\
0 & 0 & -2 & 1
\end{pmatrix} \sim
\begin{pmatrix}
1 & 0 & 0 & 0\\
0 & 2 & -1 & 0\\
0 & -1 & -3 & 0\\
0 & 0 & 0 & 1
\end{pmatrix}\sim
\begin{pmatrix}
1 & 0 & 0 & 0\\
0 & 2 & -3 & 0\\
0 & -3 & 1 & 0\\
0 & 0 & 0 & 1
\end{pmatrix}\sim
\begin{pmatrix}
1 & 0 & 0 & 0\\
0 & -7 & 0 & 0\\
0 & 0 & 1 & 0\\
0 & 0 & 0 & 1
\end{pmatrix},  ~\det G_2 =~-7,
$
hence $L'_2 \simeq [-7] \oplus [1] \oplus [1] \oplus [1]:=L(2)$,
this is a maximal anisotropic lattice,
which also
appears in \cite{vinberg07classhyplat} as an anisotropic
$2$-reflective hyperbolic lattice of rank $4$.

\vskip 1cm

(5)

\vskip 1cm

$
G_{3} =
\begin{pmatrix}
1 & 0 & 0 & 0\\
0 & 2 & -1 & -1\\
0 & -1 & 2 & -3\\
0 & -1 & -3 & 1
\end{pmatrix} \sim
\begin{pmatrix}
1 & 0 & 0 & 0\\
0 & 1 & -4 & 0\\
0 & -4 & -7 & 0\\
0 & 0 & 0 & 1
\end{pmatrix}\sim
\begin{pmatrix}
1 & 0 & 0 & 0\\
0 & 1 & 0 & 0\\
0 & 0 & -23 & 0\\
0 & 0 & 0 & 1
\end{pmatrix}, \ \det G_3 = -23,
$
\vskip 0.4cm
whence we obtain that $L'_3 \simeq [-23] \oplus [1] \oplus [1] \oplus [1]:
= L(3);$
\vskip 0.4cm
$
G_{4} =
\begin{pmatrix}
1 & 0 & -1 & 0\\
0 & 2 & 0 & -1\\
-1 & 0 & 2 & -2\\
0 & -1 & -2 & 1
\end{pmatrix} \sim
\begin{pmatrix}
1 & 0 & 0 & 0\\
0 & 1 & -2 & 0\\
0 & -2 & -3 & 0\\
0 & 0 & 0 & 1
\end{pmatrix}, \quad
L'_4 \simeq L(2);$

\vskip 0.4cm
$
G_{5} =
\begin{pmatrix}
1 & 0 & -1 & 0\\
0 & 2 & 0 & -1\\
-1 & 0 & 2 & -4\\
0 & -1 & -4& 1
\end{pmatrix} \sim
\begin{pmatrix}
1 & 0 & -1 & 0\\
0 & 1 & -4 & 0\\
-1 & -4 & -14 & 0\\
0 & 0 & 0 & 1
\end{pmatrix}
\sim
\begin{pmatrix}
1 & 0 & 0 & 0\\
0 & 1 & 0 & 0\\
0 & 0 & -31 & 0\\
0 & 0 & 0 & 1
\end{pmatrix},
$
\vskip 0.4cm
hence
$L'_5 \simeq [-31] \oplus [1] \oplus [1] \oplus [1] := L(4);$

\vskip 0.4cm
$
G_{6} =
\begin{pmatrix}
1 & 0 & -1 & 0\\
0 & 2 & -1 & 0\\
-1 & -1 & 2 & -2\\
0 &0 & -2 & 1
\end{pmatrix} \sim
\begin{pmatrix}
1 & 0 & 0 & 0\\
0 & 2 & -1 & 0\\
0 & -1 & -3 & 0\\
0 & 0 & 0 & 1
\end{pmatrix}, \ \det G_6 = -7, \ L'_6 \simeq L(2);
$

\vskip 0.4cm
$
G_{7} =
\begin{pmatrix}
1 & 0 & -1 & 0\\
0 & 2 & -1 & 0\\
-1 & -1 & 2 & -4\\
0 & 0 & -4 & 1
\end{pmatrix} \sim
\begin{pmatrix}
1 & 0 & 0 & 0\\
0 & 2 & -1 & 0\\
0 & -1 & -15 & 0\\
0 & 0 & 0 & 1
\end{pmatrix}, \ \det G_7 = -31.$
\vskip 0.4cm

Let us prove that $\ L'_7 \simeq L(4) = [-31] \oplus [1] \oplus [1] \oplus [1]$.
For this we use the Hasse principe,
which says that it suffices to verify the equivalence
of these lattices  for all $p$-adic completions.
It is clear that if $p \not = 31$, then these lattices
are unimodular and isomorphic.
It remains to consider the case $p = 31$.

We observe that $8^2  \equiv 2 \pmod{31}$,
hence
$$
G_7 \Big(e_1, \dfrac{e_2}{8}, e_3, e_4\Big) =
\begin{pmatrix}
1 & 0 & 0 & 0\\
0 & 1 & -\dfrac{1}{8} & 0\\
 0 & -\dfrac{1}{8} & -15 & 0 \\
 0 & 0 & 0 & 1\\
\end{pmatrix} \sim
\begin{pmatrix}
1 & 0 & 0 & 0\\
0 & 1 & 0 & 0 \\
0 & 0 & -15-\dfrac{1}{2} & 0 \\
0 & 0 & 0 &1 \\
\end{pmatrix},
$$
whence we obtain $\ (L'_7)_{31} \simeq (L(4))_{31}$.
Therefore, the lattice  $L'_7$ is isomorphic to the lattice $L(4)$.

\vskip 0.4cm
$
G_{8} =
\begin{pmatrix}
1 & -1 & 0 & 0\\
-1 & 2 & 0 & 0\\
0 & 0 & 2 & -3\\
0 & 0 & -3 & 1
\end{pmatrix} \sim
\begin{pmatrix}
1 & 0 & 0 & 0\\
0 & 1 & 0 & 0\\
0 & 0 & -7 & 0\\
0 & 0 & 0 & 1
\end{pmatrix}, \ \det G_8 = -7, \ L'_8 \simeq L(2);
$

\vskip 0.4cm
$
G_{9} =
\begin{pmatrix}
1 & -1 & 0 & 0\\
-1 & 2 & -1 & 0\\
0 & -1 & 2 & -4\\
0 & 0 & -4 & 1
\end{pmatrix} \sim
\begin{pmatrix}
1 & 0 & 0 & 0\\
0 & 1 & 0 & 0\\
0 & 0 & 1 & 0\\
0 & 0 & 0 & -15
\end{pmatrix}, \ \det G_{9} = -15, \ L'_{9} \simeq L(1);
$

\vskip 1cm

(6)

\vskip 1cm

$
G_{10} =
\begin{pmatrix}
1 & 0 & -1 & 0\\
0 & 2 & 0 & -1\\
-1 & 0 & 2 & -3\\
0 & -1 & -3 & 2
\end{pmatrix} \sim
\begin{pmatrix}
1 & 0 & 0 & 0\\
0 & 2 & 0 & -1\\
0 & 0 & 1 & 0\\
0 & -1 & 0 & -7
\end{pmatrix}\sim
\begin{pmatrix}
1 & 0 & 0 & 0\\
0 & -3 & 0 & 0\\
0 & 0 & 1 & 0\\
0 & 0 & 0 & 5
\end{pmatrix},
$
\vskip 0.4cm
whence it follows that
$\det G_{10} = -15, \ L'_{10} \simeq [-3] \oplus [5] \oplus [1]
\oplus [1]:=L(5);$
\vskip 0.4cm

$
G_{11} =
\begin{pmatrix}
1 & 0 & -1 & 0\\
0 & 2 & -1 & -1\\
-1 & -1 & 2 & -3\\
0 & -1 & -3 & 2
\end{pmatrix} \sim
\begin{pmatrix}
1 & 0 & 0 & 0\\
0 & 1 & 0 & 0\\
0 & 0 & 1 & -3\\
0 & 0 & -3 & -14
\end{pmatrix}, \ \det G_{11} = -23, \ L'_{11} \simeq L(3);
$

$
G_{12} =
\begin{pmatrix}
1 & 0 & -1 & 0\\
0 & 2 & -1 & -1\\
-1 & -1 & 2 & -4\\
0 & -1 & -4 & 2
\end{pmatrix} \sim
\begin{pmatrix}
1 & 0 & 0 & 0\\
0 & 1 & 0 & 0\\
0 & 0 & 1 & -4\\
0 & 0 & -4 & -23
\end{pmatrix}, \ \det G_{12} = -39,
$
\vskip 0.4cm
whence we obtain that $L'_{12} \simeq [-39] \oplus [1] \oplus [1]
\oplus [1]:= L(6);$
\vskip 0.4cm

$
G_{13} =
\begin{pmatrix}
1 & 0 & -1 & 0\\
0 & 2 & -1 & -1\\
-1 & -1 & 2 & -7\\
0 & -1 & -7 & 2
\end{pmatrix} \sim
\begin{pmatrix}
1 & 0 & 0 & 0\\
0 & 1 & 0 & 0\\
0 & 0 & 1 & -7\\
0 & 0 & -7 & -62
\end{pmatrix}, \ \det G_{13} = -111,
$
\vskip 0.4cm
whence we obtain that $L'_{13} \simeq [-111] \oplus [1] \oplus [1]
\oplus [1]:= L(7);$
\vskip 0.4cm

$
G_{14} =
\begin{pmatrix}
1 & 0 & -1 & -1\\
0 & 2 & -1 & 0\\
-1 & -1 & 2 & -1\\
-1 & 0 & -1 & 2
\end{pmatrix} \sim
\begin{pmatrix}
1 & 0 & 0 & 0\\
0 & 2 & -1 & 0\\
0 & -1 & -3 & 0\\
0 & 0 & 0 & 1
\end{pmatrix}, \ \det G_{14} = -7, \ L'_{14} \simeq L(2);
$

$
G_{15} =
\begin{pmatrix}
1 & 0 & -1 & -1\\
0 & 2 & -1 & 0\\
-1 & -1 & 2 & -3\\
-1 & 0 & -3 & 2
\end{pmatrix} \sim
\begin{pmatrix}
1 & 0 & 0 & 0\\
0 & 2 & -1 & 0\\
0 & -1 & -15 & 0\\
0 & 0 & 0 & 1
\end{pmatrix}, \ \det G_{15} = -31, \ L'_{15} \simeq L(4);
$

$
G_{16} =
\begin{pmatrix}
1 & 0 & -1 & -1\\
0 & 2 & -1 & 0\\
-1 & -1 & 2 & -5\\
-1 & 0 & -5 & 2
\end{pmatrix} \sim
\begin{pmatrix}
1 & 0 & 0 & 0\\
0 & 2 & -1 & 0\\
0 & -1 & -35 & 0\\
0 & 0 & 0 & 1
\end{pmatrix}, \ \det G_{16} = -71,$
\vskip 0.4cm
here we also observe that $12^2 = 144 \equiv 2 \pmod{71$}, whence
$\ L'_{16} \simeq [-71] \oplus [1] \oplus [1] \oplus [1]:= L(8);$

$
G_{17} =
\begin{pmatrix}
1 & 0 & -1 & -1\\
0 & 2 & -1 & -1\\
-1 & -1 & 2 & -4\\
-1 & -1 & -4 & 2
\end{pmatrix} \sim
\begin{pmatrix}
1 & 0 & 0 & 0\\
0 & 1 & -6 & 0\\
0 & -6 & -24 & 0\\
0 & 0 & 0 & 1
\end{pmatrix}, \ \det G_{17} = -60,$
\vskip 0.4cm
$L'_{17} \simeq [-60] \oplus [1] \oplus [1] \oplus [1],$ and the unique
extension of this lattice is $L(2)$.
\vskip 0.4cm

$
G_{18} =
\begin{pmatrix}
1 & -1 & 0 & 0\\
-1 & 2 & -1 & 0\\
0 & -1 & 2 & -3\\
0 & 0 & -3 & 2
\end{pmatrix} \sim
\begin{pmatrix}
1 & 0 & 0 & 0\\
0 & 1 & 0 & 0\\
0 & 0 & 1 & -3\\
0 & 0 & -3 & 2
\end{pmatrix}, \ \det G_{18} = -7, \ L'_{18} \simeq L(2);
$

$
G_{19} =
\begin{pmatrix}
1 & -1 & 0 & 0\\
-1 & 2 & -1 & 0\\
0 & -1 & 2 & -5\\
0 & 0 & -5 & 2
\end{pmatrix} \sim
\begin{pmatrix}
1 & 0 & 0 & 0\\
0 & 1 & 0 & 0\\
0 & 0 & 1 & -5\\
0 & 0 & -5 & 2
\end{pmatrix}, \ \det G_{19} = -23, \ L'_{19} \simeq L(3);
$

$
G_{20} =
\begin{pmatrix}
1 & -1 & 0 & 0\\
-1 & 2 & -1 & 0\\
0 & -1 & 2 & -7\\
0 & 0 & -7 & 2
\end{pmatrix} \sim
\begin{pmatrix}
1 & 0 & 0 & 0\\
0 & 1 & 0 & 0\\
0 & 0 & 1 & -7\\
0 & 0 & -7 & 2
\end{pmatrix}, \ \det G_{20} = -47, $
\vskip 0.4cm
whence it follows that
$\ L'_{20} \simeq [-47] \oplus [1] \oplus [1] \oplus [1] := L(9)$;

\vskip 0.4cm
$
G_{21} =
\begin{pmatrix}
1 & -1 & 0 & 0\\
-1 & 2 & -1 & -1\\
0 & -1 & 2 & -3\\
0 & -1 & -3 & 2
\end{pmatrix} \sim
\begin{pmatrix}
1 & 0 & 0 & 0\\
0 & 1 & 0 & 0\\
0 & 0 & 1 & -4\\
0 & 0 & -4 & 1
\end{pmatrix}, \ \det G_{21} = -15, \ L'_{21} \simeq L(1)$

\vskip 1cm

(7)

\vskip 1cm

$
G_{22} =
\begin{pmatrix}
2 & -1 & -1 & 0\\
-1 & 2 & 0 & 0\\
-1 & 0 & 1 & -4\\
0 & 0 & -4 & 1
\end{pmatrix} \sim
\begin{pmatrix}
1 & -1 & 0 & -4\\
-1 & 2 & 0 & 0\\
0 & 0 & 1 & -4\\
-4 & 0 & -4 & 1
\end{pmatrix}\sim
\begin{pmatrix}
1 & 0 & 0 & 0\\
0 & 1 & 0 & -4\\
0 & 0 & 1 & -4\\
0 & -4 & -4 & -15
\end{pmatrix}\sim
\begin{pmatrix}
1 & 0 & 0 & 0\\
0 & 1 & 0 & 0\\
0 & 0 & 1 & -4\\
0 & 0 & -4 & -31
\end{pmatrix},
$
\vskip 0.4cm
whence it follows that
$\det G_{22} = -47, \ L'_{22} \simeq L(9);$

\vskip 1cm

(9)

\vskip 1cm

$
G_{23} =
\begin{pmatrix}
2 & 0 & -1 & 0\\
0 & 2 & -1 & 0\\
-1 & -1 & 2 & -3\\
0 & 0 & -3 & 2
\end{pmatrix}, \ \det G_{23} = -28,\\
$
%\vskip 0.4 cm
and the lattice $L'_{23}$ has the unique extension  of index  $2$ that
can be obtained by considering this lattice in the basis
$$\Big\{\dfrac{e_1 + e_2}{2},  \dfrac{e_1 - e_2}{2},  e_3, e_4\Big\}.$$
In this basis, the matrix $G_{23}$ has the form

$
 \begin{pmatrix}
1 & 0 & -1 & 0\\
0 & 1 & 0 & 0\\
-1 & 0 & 2 & -3\\
0 & 0 & -3 & 2
\end{pmatrix}\sim
 \begin{pmatrix}
1 & 0 & 0 & 0\\
0 & 1 & 0 & 0\\
0 & 0 & 1 & 0\\
0 & 0 & 0 & -7
\end{pmatrix},
$

whence it follows that the unique extension in this case is the lattice $L(2)$.

\vskip 0.4cm

$
G_{24} =
\begin{pmatrix}
2 & 0 & -1 & -1\\
0 & 2 & 0 & 0\\
-1 & 0 & 2 & -4\\
-1 & 0 & -4 & 2
\end{pmatrix}, \ \det G_{24} = -72, \\
$
%\vskip 0.4cm
the invariant factors of the lattice $L'_{24}$ are $(1, 1, 6, 12)$. In this case,
considering the basis
$$\Big\{e_1, e_2, \dfrac{e_3 + e_4}{2},  \dfrac{e_3 - e_4}{2}\Big\},$$
 we obtain the unique extension of index $2$ that equals
$$[-1] \oplus [3] \oplus [3] \oplus [2]:= L(10).$$
\vskip 0.4cm

$
G_{25} =
\begin{pmatrix}
2 & 0 & -1 & -1\\
0 & 2 & -1 & -1\\
-1 & -1 & 2 & -3\\
-1 & -1 & -3 & 2
\end{pmatrix} , \ \det G_{25} = -60,
$
\vskip 0.4cm
in this case we also use the  basis
$\Big\{\dfrac{e_1 + e_2}{2},  \dfrac{e_1 - e_2}{2},  e_3, e_4\Big\}$ and
obtain the unique extension  of index  $2$ that equals $L(1)$.
\vskip 0.4cm

$
G_{26} =
\begin{pmatrix}
2 & -1 & 0 & 0\\
-1 & 2 & -1 & -1\\
0 & -1 & 2 & -4\\
0 & -1 & -4 & 2
\end{pmatrix} , \ \det G_{26} = -60,
$
\vskip 0.4cm
using the basis
$$\Big\{\dfrac{e_1 + e_4}{2},  e_2, e_3,  \dfrac{e_1 - e_4}{2}\Big\}$$
we obtain the unique (up to an isomorphism) extension of index  $2$ that
equals $L(1)$.
\vskip 0.4cm

$
G_{27} =
\begin{pmatrix}
2 & -1 & 0 & -1\\
-1 & 2 & -1 & 0\\
0 & -1 & 2 & -3\\
-1 & 0 & -3 & 2
\end{pmatrix}, \ \det G_{27} = -28,
$
in this case we consider the basis
$$
\Big\{\dfrac{e_1 + e_2 + e_3 + e_4}{2},  \dfrac{e_1 + e_2 - e_3
- e_4}{2},  e_3, e_4\Big\},
$$
which gives the unique extension  equal to $L(2)$.
\end{small}

\subsection{The list of maximal anisotropic lattices-pretendents}\label{sp-prom-pretend}

\begin{small}
In this general list we collect all intermediate maximal anisotropic lattices.
\vskip 0.5 cm

\begin{tabular}{|c|c|c|c|c|c|}
\hline
L(k) & $L$ & Invariant factors & Discriminant\\
\hline
$L(1)$ & $[-15] \oplus [1] \oplus [1] \oplus [1]$
& $(1, 1, 1, 15)$ & $-15$ \\
\hline
$L(2)$ & $[-7] \oplus [1] \oplus [1] \oplus [1]$ & $(1, 1, 1, 7)$ & $-7$ \\
\hline
$L(3)$ & $[-23] \oplus [1] \oplus [1] \oplus [1]$ & $(1, 1, 1, 23)$ & $-23$ \\
\hline
$L(4)$ & $[-31] \oplus [1] \oplus [1] \oplus [1]$ & $(1, 1, 1, 31)$ & $-31$ \\
\hline
$L(5)$ & $[-3] \oplus [5] \oplus [1] \oplus [1]$ & $(1, 1, 1, 15)$ & $-15$ \\
\hline
$L(6)$ & $[-39] \oplus [1] \oplus [1] \oplus [1]$ & $(1, 1, 1, 39)$ & $-39$ \\
\hline
$L(7)$ & $[-111] \oplus [1] \oplus [1] \oplus [1]$
& $(1, 1, 1, 111)$ & $-111$ \\
\hline
$L(8)$ & $[-71] \oplus [1] \oplus [1] \oplus [1]$ & $(1, 1, 1, 71)$ & $-71$ \\
\hline
$L(9)$ & $[-47] \oplus [1] \oplus [1] \oplus [1]$ & $(1, 1, 1, 47)$ & $-47$ \\
\hline
$L(10)$ & $[-1] \oplus [3] \oplus [3] \oplus [2]$ & $(1, 1, 3, 6)$ & $-18$ \\
\hline

\end{tabular}
\end{small}

\section{Verification of $1.2$-reflectivity}\label{issled-refl}

It remains to test on reflectivity a very small number of lattices.
The lattices $L(1)$ and $L(2)$ are maximal $2$-reflective
anisotropic lattices, i.e., they are also $1.2$-reflective.
The $2$-reflectivity of this lattices
was proved in \cite{vinberg07classhyplat}.

Non-reflectivity of the lattice $L(3)$ was proved in  \cite{mark13primelat},
non-reflectivity of the lattice $L(6)$ was proved in the dissertation
\cite{mcleod13arithhrgdisser}, where one can also find the proof
of the fact that the lattices
$L(4)$, $L(5)$, $L(7)-L(9)$ are non-reflective.
The non-reflectivity of these five lattices
follows from the Bugaenko theorem (\ref{bugaenko}) along with the
absence of analogous lattices of lower dimension in Nikulin's list.

\begin{proposition}
The lattice $L(5) = [-3] \oplus [5] \oplus [1] \oplus [1]$ is
reflective, but not $1.2$-reflective.
\end{proposition}
\begin{proof}
We first observe that   the scalar
product corresponding to this lattice has the following form:
$$
(x, y) = - 3 x_0 y_0 + 5 x_1 y_1 + x_2 y_2 + x_3 y_3,
$$
where $x = (x_0, x_1, x_2, x_3), \ y = (y_0, y_1, y_2, y_3)
\in \mathbb{R}^{3,1}$.

Following Vinberg's algorithm, we pick the basic point
$v_0 = (1; 0, 0, 0)$ and normals for the sides of the fundamental polyhedral cone:

$a_1 = (0; 0, 0, -1),  \ (a_1, a_1) = 1$;

$a_2 = (0; 0, -1, 1), \ (a_2, a_2) = 2$;

$a_3 = (0; -1, 0, 0), \ (a_3, a_3) = 5$.

It is worth noting here that when choosing a $k$-root, we should take into account
that $k$ must divide the double largest
invariant factor,  i.e., $k$ must be
a divisor of $30$. Every next $k$-root $a = (x_0, x_1, x_2, x_3)$
is determined by the conditions
$$
10 x_1 \equiv 2x_2 \equiv 2x_3 \equiv 0 \pmod k, \quad x_0 > 0,
\quad (a, a_j) \le 0, \quad
\dfrac{|(a, v_0)|}{\sqrt{(a, a)}} = \dfrac{3x_0}{\sqrt{k}}=
\min.
$$
We observe that if $k$ is divisible by $3$, then
it follows from the  conditions indicated above
that the number $k/3$ gives the remainder $2$ when divided by $3$,
and if $k = 5t$, then
$t$ is a square by modulo~$5$. This yields that
$k$ can equal only $1, 2, 5, 6$.

Clearly, the minimal value of $\dfrac{|(a, v_0)|}{\sqrt{(a, a)}}$
is achieved when the   length of the root we are picking is maximal and
the value $x_0$ is minimal. We find the fourth root

$a_4 = (1; 0, 3, 0), \ (a_4, a_4) = 6$.

The Coxeter  diagram for the first four roots does not determine
a polyhedron of a finite volume.

The fifth root must  now satisfy the conditions
$(a_5, a_j) \le 0$ whenever  $j \le 4$, hence
$$
x_0 \ge x_2 \ge x_3 \ge 0, \quad x_0 \ge 1.
$$

If $k = 5$, then it is not difficult to see that $x_0$ must be divisible by $5$.
It is also clear
that $\dfrac{1}{\sqrt{2}} < \dfrac{2}{\sqrt{6}}$,
so the nearest fifth root is the root

$a_5 = (1; 1, 0, 0), \ (a_5, a_5) = 2$.

It is clear that the Coxeter diagram at this step still does not determine
a polyhedron of a finite volume.

The sixth root must satisfy the conditions
$$
3x_0 \ge 5x_1, \quad x_0 \ge x_2 \ge x_3 \ge 0, \quad x_0 \ge 1.
$$

A quick analysis of these cases shows that

$a_6 = (2; 1, 2, 2), \ (a_6, a_6) = 1$.

The Coxeter diagram of the first six roots also does not determine a polyhedron
of a finite volume, so we should find the next root.

The seventh root must satisfy the  additional condition
$$
-6x_0 + 5 x_1 + 2x_2 + 2 x_3 \le 0.
$$

At this step the analysis of various cases becomes rather lengthy and yields
the seventh root

$a_7 = (10; 6, 10, 5), \ (a_7, a_7) = 5$.

For the obtained seven roots the Coxeter diagram has the following form:

\begin{center}

{\scshape Picture 1.}

\includegraphics{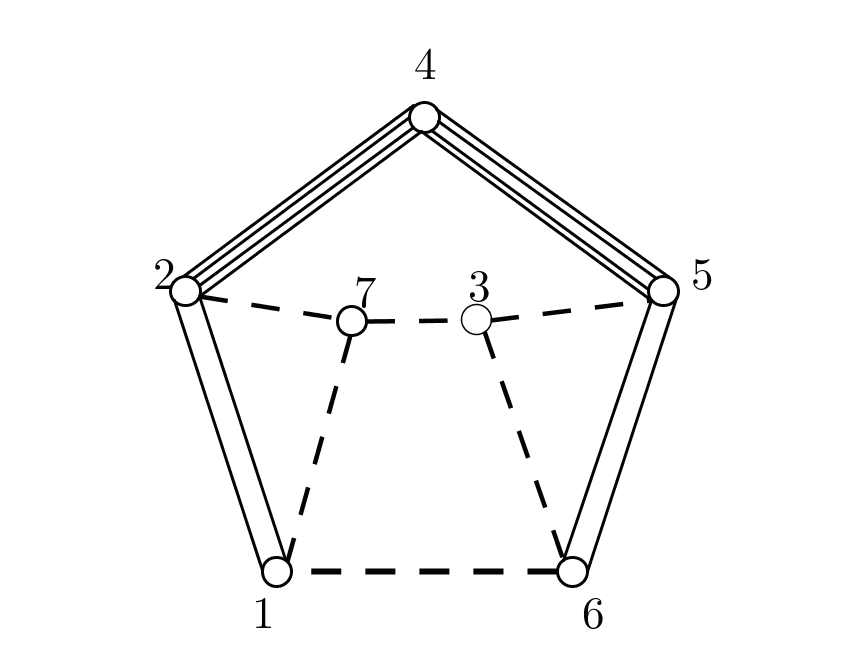}

\end{center}

It determines a bounded three-dimensional  Coxeter polyhedron.
We observe that the roots $a_3$, $a_4$, $a_7$ determine the group
generated by ``bad'' reflections. This group is infinite, since
the corresponding subdiagram contains a dotted edge. Therefore, the lattice $L(5)$ is
reflective, but not $1.2$-reflective.
\end{proof}

\begin{proposition}
The lattice $L(10) = [-1] \oplus [3] \oplus [2] \oplus [2]$ is
reflective, but not $1.2$-reflective.
\end{proposition}
\begin{proof}
The corresponding scalar
product has the following form:
$$
(x, y) = - x_0 y_0 + 3 x_1 y_1 + 2x_2 y_2 + 2x_3 y_3,
$$
where $x = (x_0, x_1, x_2, x_3), \ y = (y_0, y_1, y_2, y_3)
\in \mathbb{R}^{3,1}$.

Following Vinberg's algorithm, we pick the basic point
$v_0 = (1; 0, 0, 0)$ and normals for the sides of the fundamental polyhedral cone:

$a_1 = (0; 0, 0, -1),  \ (a_1, a_1) = 2$;

$a_2 = (0; 0, -1, 0), \ (a_2, a_2) = 3$;

$a_3 = (0; -1, 2, 0), \ (a_3, a_3) = 6$.

We observe that this lattice can have $k$-roots only for
$k = 1, 2, 3, 6$. This can be easily verified, since $k$
must be a divisor of $12$, and also
$k$ is not divisible by $4$, because otherwise
one could reduce this root by $2$.

Further, if some root $a = (x_0, x_1, x_2, x_3)$ has
the square divisible by $3$, then it is readily seen that
$x_0 \equiv x_3 \equiv 0 \pmod 3.$

Due to these conditions, it is easy to find the forth and the fifth roots:

$a_4 = (1; 1, 0, 0), \ (a_4, a_4) = 2$,

$a_5 = (1; 0, 0, 1), \ (a_5, a_5) = 1$.

The Coxeter diagram does not determine
a polyhedron of a finite volume, so we find the sixth
root

$a_6 = (6; 2, 2, 3), \ (a_6, a_6) = 6$.

For the obtained six roots the Coxeter diagram has the following form:

\begin{center}

{\scshape Picture 2.}

\includegraphics{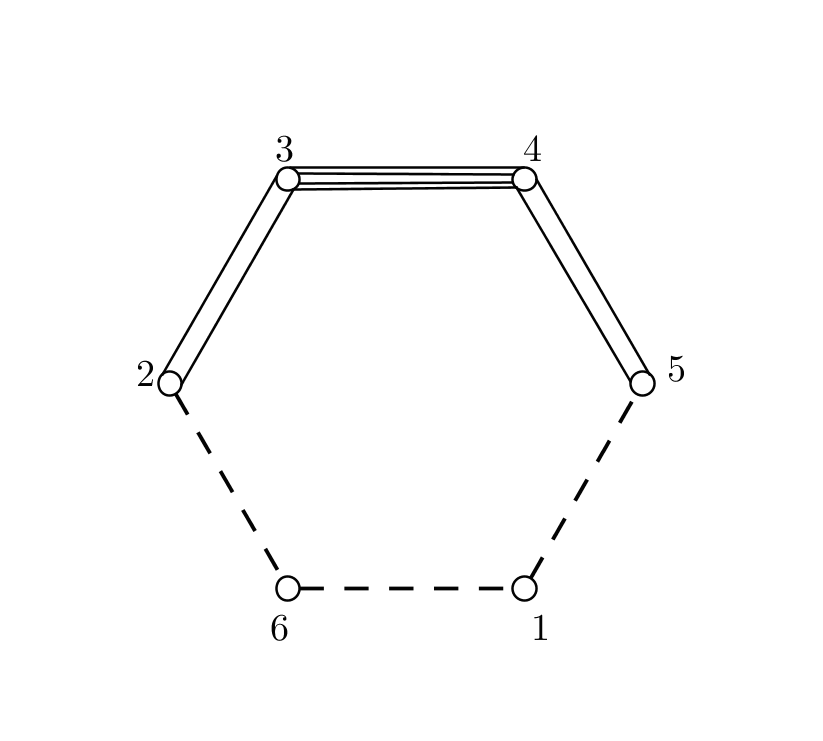}

\end{center}

It determines a bounded three-dimensional  Coxeter polyhedron that
is a tetrahedron with two cropped vertices.
We observe that the roots $a_2$, $a_3$ and $a_6$ determine the group
generated by ``bad'' reflections. This group is infinite, since
the corresponding subdiagram contains a dotted edge. Therefore, the
lattice $L(10)$ is reflective, but not $1.2$-reflective.
\end{proof}

Thus, only two of the ten maximal anisotropic lattices picked
in the process of solving our problem are $1.2$-reflective. These lattices are
$L(1)$ and $L(2)$, hence Theorem \ref{osn-res} is now proven.

\end{document}